\documentclass[10pt,oneside, reqno]{amsart}

\usepackage[colorlinks,linkcolor=blue,citecolor=blue]{hyperref}
\usepackage{graphics}
\usepackage{enumitem}
\usepackage{marginnote}	
\usepackage{latexsym, amssymb, amsmath, amsthm}
\usepackage{mathrsfs}
\usepackage{tikz}
\usepackage{tikz-cd}
\usepackage[only,llbracket,rrbracket]{stmaryrd}
\usetikzlibrary{decorations.pathmorphing}

\tikzset{
	squiggly/.style={decorate, decoration={snake, amplitude=1mm, segment length=6mm}},
}

\newtheorem{theorem}{Theorem}[section]
\newtheorem{lemma}[theorem]{Lemma}
\newtheorem{proposition}[theorem]{Proposition}
\newtheorem{corollary}[theorem]{Corollary}

\theoremstyle{definition}
\theoremstyle{definition}
\theoremstyle{definition}\newtheorem{remark}[theorem]{Remark}

\def\ga{\mathfrak{a}}
\def\gb{\mathfrak{b}}

\def\gm{\mathfrak{m}}
\def\gp{\mathfrak{p}}

\def \CA{\mathcal{A}}

\def \CC{\mathcal{C}}
\def \CD{\mathcal{D}}

\def \CJ{\mathcal{J}}

\def \Prim{{\rm Prim}}
\def\Max{{\rm Max}}

\def \Spec{{\rm Spec}}
\def \ann{{\rm ann}}
\def \Frac{{\rm Frac}}

\def \mK{\Bbbk}
\def \Z{\mathbb{Z}}
\def \N{\mathbb{N}}

\def \d{\delta}
\def \s{ \sigma }
\def \l{\lambda}

\begin{document}

	\author{Tao Lu}
	
	
	
	\address{School of Mathematical Science, Yangzhou University, Yangzhou, 225002, China}
	
	\email{taolu@yzu.edu.cn}
	
	\begin{abstract} 
		We study the ring-theoretic structure and representation theory of the super Jordan plane $\CJ$ over fields of characteristic different from $2$. We prove that $\CJ$ is prime and classify its prime, primitive, and maximal ideals. We determine its classical ring of quotients and classify the finite-dimensional simple modules, while relating infinite-dimensional simple modules to those of the first Weyl algebra. Our approach is based on showing that a localization of $\CJ$ is a matrix algebra over a localization of the first Weyl algebra. 
	\end{abstract}
	
	\title{Prime spectrum and representations of the super Jordan plane}
	\maketitle
	

\section{Introduction}

The Jordan plane and the super Jordan plane are two Nichols algebras that play an important role in the classification of pointed Hopf algebra over abelian groups with finite Gelfand--Kirillov dimension; see, for instance, \cite{Andruskiewitsch-Angiono-Heckenberger,Andruskiewitsch-Angiono-Heckenberger-1,Andruskiewitsch-Angiono-Heckenberger-2}. The Jordan plane is the algebra generated by $\mathbf{x}$ and $\mathbf{y}$ subject to the relation $\mathbf{y} \mathbf{x}-\mathbf{x}\mathbf{y}=-\frac{1}{2}\mathbf{x}^2$.  It is a Noetherian domain of Gelfand--Kirillov dimension $2$, whose structure and representation theory are well understood; see, for example, \cite{Artin-Shelter,Iyudu,Shirikov}.

The super Jordan plane, introduced in \cite{Andruskiewitsch-Angiono-Heckenberger}, is also a Noetherian algebra of Gelfand--Kirillov dimension $2$. In contrast to the Jordan plane, it is not a domain, and its structure exhibits phenomena that do not arise in the classical case. This makes the super Jordan plane a natural object for further investigation from both ring-theoretic and representation-theoretic perspectives.

The super Jordan plane $\CJ$ is the $\mK$-algebra generated by $x$ and $y$ subject to the defining relations
\begin{equation*}
x^2=0,\qquad y^2x -xy^2-xyx=0. 
\end{equation*}
Several aspects of its structure have been studied in recent years.
In characteristic zero, finite-dimensional simple modules and indecomposable modules of dimensions $2$ and $3$ were classified in \cite{Andruskiewitsch-Bagio-Della Flora-Flores-1}. The Hochschild homology and cohomology were computed in \cite{Reca-Solotar}. In \cite{Andruskiewitsch-Dumas}, it was shown that $\CJ$, viewed as a $\Z/2\Z$-graded algebra, is super-prime and admits a super-simple super-Artinian ring of fractions; moreover, the groups of superalgebra and braided Hopf algebra automorphisms were determined. A classification of finite-dimensional simple modules over a suitable bosonization of the super Jordan plane was given in \cite{Andruskiewitsch-Bagio-Della Flora-Flores-2}. The Drinfeld double of a suitable bosonization of $\CJ$ was studied in \cite{Andruskiewitsch-Pena Pollastri}. Liftings of both the Jordan plane and the super Jordan plane were computed in \cite{Andruskiewitsch-Angiono-Heckenberger-1,Andruskiewitsch-Angiono-Heckenberger-2}. 

Despite this progress, the ordinary ring-theoretic structure of $\CJ$ remains only partially understood. In particular, a question left open in \cite{Andruskiewitsch-Dumas} is whether $\CJ$ is prime as an (ungraded) algebra. The main goal of this paper is to resolve this question and to determine the prime spectrum, the classical ring of quotients, and the representation theory of $\CJ$.

Our main result is a classification of the prime, primitive, and maximal ideals of $\CJ$ over fields of characteristic different from $2$. In particular, this shows that $\CJ$ is a prime algebra and leads to a classification of the simple $\CJ$-modules.
In characteristic zero,  all finite-dimensional simple modules are one-dimensional, while a simple module is infinite-dimensional if and only if the element $s:=xy+yx$ acts invertibly. In that case, the module extends naturally to a simple module over the localization $\CJ_s$, whose representation theory is governed by that of the first Weyl algebra. In characteristic $p>2$, all simple modules are finite-dimensional of dimension either $1$ or $2p$, and we give explicit constructions of all such modules.

We also determine the classical ring of quotients $Q(\CJ)$. It was shown in \cite[\S 4]{Andruskiewitsch-Dumas} that $Q(\CJ)$ admits a structure of superalgebra and is a super-simple super-Artinian ring, although no explicit description was given. In characteristic zero we prove that $Q(\CJ)  \cong M_2(D_1(\mK))$, a matrix algebra over the Weyl skewfield $D_1(\mK)=\Frac (A_1)$, while in characteristic $p>2$ we show that $Q(\CJ)$ is a central simple algebra of dimension $(2p)^2$ over its center $\mK(s^p,y^{2p})$. In positive characteristic, we further prove that $\CJ$ is a prime PI algebra of PI degree~$2p$, and its localization $\CJ_s$ is an Azumaya algebra of rank $(2p)^2$ over its center.

A key ingredient in our approach is the observation that the localization $\CJ_s$ of $\CJ$ at the powers of the normal element $s$ is isomorphic to a matrix algebra over a localization $\CA$ of the first Weyl algebra; more precisely, $\CJ_s \cong M_2(\CA)$. Thus the algebras $\CJ_s$ and $\CA$ are Morita equivalent, and in particular  the representation theory of $\CJ_s$ is determined by that of $\CA$.  This observation allows many properties of $\CJ_s$ and $\CJ$ to be deduced directly from well-known facts about Weyl algebras and matrix rings, and provides a structural explanation for several previously known results, including the description of the center obtained in \cite{Reca-Solotar,Andruskiewitsch-Dumas}.

The paper is organized as follows. In Section \ref{Localization}, we recall basic properties of the super Jordan plane and study the localization $\CJ_s$. In Section \ref{SupJordan-zero}, we investigate $\CJ$ over a field of characteristic zero, focusing on its center, classical ring of quotients, prime and primitive ideals, and simple modules. In Section \ref{SupJordan-p}, we study $\CJ$ in positive characteristic, where we determine its center, PI degree, and classical ring of quotients, and classify its prime ideals, Azumaya locus, and simple modules.
 
 Throughout the paper, all modules are left modules. The base field $\mK$ is assumed to have characteristic different from $2$, with additional assumptions imposed when needed. We write $\mK^*=\mK\setminus\{0\}$ and $\N=\{0, 1, 2, \dots \}$. The center of a ring $R$ is denoted by $Z(R)$.

\section{A localization of the super Jordan plane}  \label{Localization}  

In this section, we recall basic properties of the super Jordan plane and study the localization of $\CJ$ at the powers of the normal element $s$, which plays a central role in our main results. We also determine the prime ideals of the factor algebra $\CJ/s\CJ$.
\subsection{The super Jordan plane}
 The super Jordan plane $\CJ$ is the $\mK$-algebra generated by $x$ and $y$ subject to the defining relations
\begin{equation*}
x^2=0, \qquad y s-sy-xs=0, \quad {\rm where}\,\, s:=xy+yx. 
\end{equation*}
These relations imply
\begin{equation*}
sx=xs, \qquad sy^2-y^2 s=-s^2. 
\end{equation*}
In particular,  the subalgebra of $\CJ$ generated by $s$ and $y^2$ is isomorphic to the Jordan plane. A straightforward induction also shows that for all $n\geq 1$,
\begin{equation} \label{y2s}  
(y^2-s)^n=\sum_{i=0}^n(-1)^i \binom{n}{i} i! y^{2(n-i)}s^i. 
\end{equation}

The algebra $\CJ$ is naturally $\N$-graded with $\deg\,x=\deg\,y=1$. By \cite{Andruskiewitsch-Angiono-Heckenberger},  $\CJ$ has Gelfand--Kirillov dimension 2, and the elements $x^i s^j y^k$ ($i=0,1$; $j,k\in \N$) form a PBW basis of $\CJ$. Furthermore, by \cite[Corollary 1.4]{Andruskiewitsch-Dumas} and \cite{Reca-Solotar}, $\CJ$ is Noetherian and has infinite global dimension. It was also shown in \cite[Proposition 1.3]{Andruskiewitsch-Dumas} that $\CJ$ can be presented as an Ore extension.

The super Jordan plane arises as a Nichols algebra associated to the braided vector space $(V,c)$ with basis $\{x,y\}$ and braiding given by
\begin{equation*}
\begin{aligned}
c(x \otimes x)&=-x\otimes x, &c(y \otimes x)&=-x\otimes y,\\
c(x \otimes y)&=-(y+x)\otimes x, & c(y \otimes y)&=(-y+x)\otimes y. 
\end{aligned}
\end{equation*}
This algebra is one of the basic examples of Nichols algebras associated to non-diagonal braidings.

\subsection{The Weyl algebra}
The first Weyl algebra $A_1$ is the associative $\mK$-algebra generated by $\partial$ and $x$ subject to the defining relation $\partial x-x \partial=1$. If ${\rm char}\,\mK=0$, then $A_1$ is a simple Noetherian domain with center $Z(A_1)=\mK$. Its skewfield of fractions $D_1(\mK):=\Frac(A_1)$ is called the Weyl skewfield. If ${\rm char}\,\mK=p>0$, then $Z(A_1)=\mK[\partial^p, x^p]$, and $A_1$ is an Azumaya algebra over its center.   

A standard normal-ordering identity states that for all $n\geq 1$, 
\begin{equation*}
(x \partial)^n=\sum_{k=1}^n S(n,k)x^k \partial^k
\end{equation*}
where $S(n,k)$ denotes the Stirling numbers of the second kind. In characteristic $p$, all intermediate coefficients vanish modulo $p$, yielding
\begin{equation} \label{normp} 
(x \partial)^p = x^p \partial^p +x\partial. 
\end{equation}

\subsection{The localization $\CJ_s$}
An element $a$ in a ring $R$ is said to be \emph{normal} if $aR=Ra$.  The element $s$ is normal in $\CJ$, since $sx=xs$ and $sy=(y-x)s$. 
We denote by $\CJ_s$ the localization of $\CJ$ at the powers of $s$ and set
\begin{equation*}
y':=ys^{-1}.
\end{equation*}
A direct computation shows that the following relations hold in $\CJ_s$:
\begin{equation*}
 xy'+y'x=1, \qquad y's-sy'=x. 
\end{equation*}

Let $\CD:=\mK[x]/\langle x^2 \rangle$ denote the algebra of dual numbers. Then $\CJ_s$ can be presented as an Ore extension of the form $\CJ_s=\CD[s^{\pm 1}][y';\s,\d]$, where $\s$ is the automorphism and $\d$ is the $\s$-derivation of $\CD[s^{\pm 1}]$ defined by $\s(x)=-x$, $\s(s)=s$; $\d(x)=1$ and $\d(s)=x$. In particular, $\CJ_s$ is a Noetherian algebra, and the set $\{x^is^jy^{\prime\, k}\mid i=0,1, j\in \Z, k\in \N\}$ forms a PBW basis of $\CJ_s$.

A set of $n\times n$ \emph{matrix units} of a ring $T$ is a collection $\{e_{ij} \mid 1 \leq i,j \leq n \} \subset T$ satisfying 
\begin{equation*}
\sum_{i=1}^n e_{ii}=1 \quad \text{and} \quad  e_{ij}e_{k\ell}=\delta_{jk}e_{i \ell}
\end{equation*}
for all $i,j,k,\ell$. The existence of such a set forces $T$ to be a full matrix ring over a suitable ring. Indeed, if we define $R=\{\sum_{k=1}^n e_{k1} a e_{1k} \mid a\in T \}$, then $R$ is a subring of $T$ and $T\cong M_n(R)$. Explicitly, the map
\begin{equation} \label{Matrixiso}  
\varphi:T\longrightarrow M_n(R),\qquad
a\longmapsto (r_{ij}),
\quad\text{where}\quad
r_{ij}=\sum_{k=1}^n e_{ki}ae_{jk},
\end{equation}
is a ring isomorphism (see, for instance, \cite[Proposition 1.1.3]{Rowen}).

The next lemma identifies $\CJ_s$ as a $2\times 2$ matrix algebra over a localization of the first Weyl algebra. 

\begin{lemma} \label{12Jan26}  
	Let $\CA$ be the subalgebra of $\CJ_s$ generated by $y^{\prime \,2}$ and $s^{\pm 1}$. Then $\CA$ is a localization of the first Weyl algebra and 
	$\CJ_s \cong M_2(\CA)$.   More precisely, the map 
	\begin{equation*}
	\varphi: \CJ_s \longrightarrow M_2(\CA), \quad x \longmapsto \begin{bmatrix}
	0 & 0 \\ 1 & 0
	\end{bmatrix}, \quad s^{\pm 1} \longmapsto \begin{bmatrix}
	s^{\pm 1} & 0 \\ 0 & s^{\pm 1}
	\end{bmatrix}, \quad y' \longmapsto \begin{bmatrix}
	0 & 1 \\ y^{\prime \,2} & 0
	\end{bmatrix},
	\end{equation*}
	defines an algebra isomorphism.
\end{lemma}
\begin{proof}
	First observe that $[y^{\prime \,2},s]=[y',s]y'+y'[y',s]=xy'+y'x=1$.  Therefore the subalgebra generated by $y^{\prime \,2}$ and $s$ is isomorphic to the first Weyl algebra. Since $s$ is invertible in $\CA$, it follows that $\CA$ is a localization of the first Weyl algebra.
	
	Set $e:=y'x=ys^{-1}x \in \CJ_s$. Then $e^2=e$, so $e$ is an idempotent.  A direct computation shows that
	\begin{equation} \label{idemp} 
	ex=0, \quad xe=x, \quad ey'= y'(1-e), \quad y'e=(1-e)y',  \quad es=se.
	\end{equation}
	In particular, $e$ commutes with every element of $\CA$. Moreover, since $xy^{\prime \,2}=y^{\prime \,2}x$ and $xs=sx$, the element $x$ also commutes with  $\CA$.  
	Define
	\begin{equation*}
		\begin{aligned}
				e_{11}:&=e, \quad e_{12}:=ey', \\ e_{21}:&=x, \quad e_{22}:=1-e. 
		\end{aligned}
	\end{equation*}  
	Using \eqref{idemp}, one checks that $\{e_{ij}\mid 1 \leq i, j \leq 2\}$ is a set of matrix units in $\CJ_s$. Hence $\CJ_s \cong M_2(R)$ for a suitable ring $R$. From the construction (see \eqref{Matrixiso}), we may take $R=\{e_{11} a e_{11}+e_{21} a e_{12}\mid a \in \CJ_s\}$. 
	
	From the PBW basis, the algebra $\CJ_s$ is a free $\CA$-module of rank 4 with decomposition
	$$\CJ_s=\CA \oplus \CA x \oplus \CA y' \oplus \CA e.$$
	Thus every $a \in \CJ_s$ can be uniquely written as $a=a_1+a_2 x+ a_3 y'+a_4 e$ where $a_i \in \CA$. 	
	Using \eqref{idemp}, we obtain $exe=0$, $ey'e=0$, and $xey'=xy'=(1-e)$. Since $e$ and $x$ commute with $\CA$, it follows that $e_{11}ae_{11}=eae=a_1 e+a_4 e$, and $e_{21}ae_{12}=xaey'=a_1(1-e)+a_4(1-e)$. Hence $e_{11}ae_{11}+e_{21}ae_{12}=a_1+a_4$. Since $a_1$ and $a_4$ vary freely in $\CA$, we conclude that $R=\CA$. Therefore, $\CJ_s \cong M_2(\CA)$.  The explicit form of the isomorphism  follows from \eqref{Matrixiso}. 
\end{proof}

\subsection{Prime ideals of $\CJ/s\CJ$}
Recall that a proper ideal $\gp$ of a ring $R$ is said to be \emph{prime} if,  for all ideals $\ga$ and $\gb$ of $R$, the inclusion $\ga \gb \subseteq \gp$ implies that $\ga \subseteq \gp$ or $\gb \subseteq \gp$. 
A prime ideal $\gp$ is called \emph{completely prime} if the factor ring $R/\gp$ is a domain. An ideal $\gp$ of a ring $R$ is called a \emph{primitive ideal} if it is the annihilator of some simple $R$-module. It is well known that primitive ideals are prime and maximal ideals are primitive.

Define $\Lambda:=\CJ/s\CJ$. Then 
\begin{equation*}
\Lambda \cong \mK \langle x, y\mid xy+yx=0, \,\,x^2=0 \rangle. 
\end{equation*}
As a vector space, $\Lambda=\mK[y] \oplus x \mK[y]$. The following proposition describes the prime ideals of $\Lambda$. 
\begin{proposition} \label{24Jan26}  
	Assume that $\mK$ is an algebraically closed field of arbitrary characteristic. Then 
	the prime spectrum of $\Lambda$ is  
	\begin{equation*}
	\Spec(\Lambda)=\{ \langle x \rangle\} \,\, \cup \,\,\{\langle x, \,y-\alpha \rangle \mid \alpha \in \mK   \}. 
	\end{equation*}
\end{proposition}
\begin{proof}
	Let $P$ be a prime ideal of $\Lambda$. Since $\langle x\rangle^2 =\langle x^2 \rangle =\langle 0 \rangle \subseteq P$, the primeness of $P$ implies that $\langle x \rangle \subseteq P$. Hence every prime ideal of $\Lambda$ contains $x$.  It follows that the prime ideals of $\Lambda$ are in bijection with the prime ideals of the factor algebra $\Lambda /\langle x\rangle \cong \mK[y]$. Since $\mK$ is algebraically closed, we have $\Spec(\mK[y])=\{ \langle 0 \rangle\} \,\cup\,\{\langle y-\alpha \rangle \mid \alpha \in \mK \}$. The result follows by lifting these ideals to $\Lambda$. 
\end{proof}

\section{The super Jordan plane in characteristic zero} \label{SupJordan-zero} 

Throughout this section, we work over a field $\mK$ of characteristic zero.
We study the super Jordan plane $\CJ$, with emphasis on its
center, classical ring of quotients, prime and primitive ideals, and simple modules.
\subsection{Center and quotient ring of $\CJ$}
Let $\CC_R$ denote the set of regular elements of a ring $R$. Then $\CC_R$ is a multiplicative set. If $\CC_R$ is a left (resp. right) Ore set, the localization $Q_{\ell}(R):=\CC_R^{-1}R$ (resp. $Q_{r}(R):=R \CC_R^{-1}$) is called the left (resp. right) quotient ring of $R$. If $\CC_R$ is an Ore set (that is, both a left and a right Ore set), then $Q_{\ell}(R)$ and $Q_r(R)$ are isomorphic, and the ring $Q(R):=\CC_R^{-1}R=R \CC_R^{-1}$ is called the \emph{classical quotient ring} of $R$.

 Since $\CJ$ is Noetherian, the set of regular elements satisfies the Ore condition.
We now describe the center, and classical quotient ring of $\CJ$ in characteristic zero.
\begin{proposition} \label{21Jan26}  
	Assume ${\rm char}\,\mK=0$. 
	\begin{enumerate}
		\item The centers of $\CJ_s$ and $\CJ$ are trivial: $Z(\CJ_s)=Z(\CJ)=\mK$. 
		\item The algebra $\CJ_s$ is simple, and hence prime. 
		\item The classical quotient ring of $\CJ$ is $Q(\CJ)  \cong M_2(D_1(\mK))$, a central simple algebra over the Weyl skewfield $D_1(\mK)$. 
	\end{enumerate}
\end{proposition}
\begin{proof}
	(1) Since $Z(\CA)=\mK$, Lemma~\ref{12Jan26} implies that $Z(\CJ_s)=\mK$. Thus, $Z(\CJ)=\CJ\cap Z(\CJ_s)=\mK$. 
	
	(2) The algebra $\CA$ is simple, hence the matrix algebra $\CJ_s \cong M_2(\CA)$ is simple, and therefore prime.
	
    (3) By Lemma~\ref{12Jan26} and \cite[Corollary 3.1.5]{MR}, we have 
$Q(\CJ_s) \cong M_2\bigl(Q(\CA)\bigr)
\cong M_2\bigl(D_1(\mK)\bigr).$
Since localization at regular elements yields $Q(\CJ)=Q(\CJ_s)$, the result follows. 
\end{proof}

For algebraically closed fields $\mK$ of characteristic zero, the equality $Z(\CJ)=\mK$ was proved in \cite[Theorem 3.12]{Reca-Solotar} using Hochschild cohomology. Indeed, the zeroth Hochschild cohomology ${\rm H}^0(\CJ,\CJ)$ coincides with the center $Z(\CJ)$.

\subsection{Prime ideals of $\CJ$}

The following theorem gives an explicit description of the prime, completely prime, primitive and maximal ideals of the super Jordan plane $\CJ$ in characteristic zero.
\begin{theorem} \label{A24Jan26}   
	Assume that $\mK$ is algebraically closed of characteristic zero. 
	\begin{enumerate}
		\item The prime spectrum of $\CJ$ is 
		\begin{equation*}
		\Spec(\CJ)=\{ \langle 0 \rangle, \, \langle s, \,x \rangle \} \,\,\cup\,\,\{ \langle s, \, x,\, y-\alpha \rangle \mid \alpha \in \mK \}. 		
		\end{equation*}
		
		\item Every nonzero prime ideal of $\CJ$ is completely prime. 
		
		\item The primitive spectrum of $\CJ$ is $\Prim(\CJ)=\{ \langle 0 \rangle \} \,\,\cup\,\,\{ \langle s, \, x,\, y-\alpha \rangle \mid \alpha \in \mK \}$. 
		
		\item The maximal spectrum of $\CJ$ is $\Max(\CJ)=\{ \langle s, \, x,\, y-\alpha \rangle \mid \alpha \in \mK \}$. 
	\end{enumerate}
\end{theorem}
\begin{proof}
	(1) Let $\Spec(\CJ, s)$ (resp. $\Spec_s(\CJ)$) denote the set of prime ideals of $\CJ$ containing (resp. not containing) $s$. Then $\Spec(\CJ)=\Spec(\CJ, s) \cup \Spec_s(\CJ)$. Prime ideals of $\CJ$ containing $s$ are in bijection with prime ideals of the factor algebra $\Lambda=\CJ/s\CJ$. By Proposition \ref{24Jan26},
	\begin{equation*}
	\Spec(\CJ, s)= \{ \langle s, \,x \rangle\} \,\, \cup \,\,\{\langle s,\,x, \,y-\alpha \rangle \mid \alpha \in \mK   \}. 
	\end{equation*}
	Since $s$ is a normal element of $\CJ$, prime ideals of $\CJ$ not containing $s$ are in bijection with prime ideals of the localization $\CJ_s$. By Proposition \ref{21Jan26}(2), the algebra $\CJ_s$ is simple, and hence $\Spec(\CJ_s)=\{ \langle 0 \rangle \}$. It follows that $\Spec_s(\CJ)=\{\langle 0 \rangle \}$, which yields the stated description of $\Spec(\CJ)$. 
	
	(2) The zero ideal $\langle 0 \rangle$ is not completely prime, since $\CJ$ is not a domain. On the other hand,
	\begin{equation*}
	\CJ/\langle s, \, x \rangle \cong \mK[y] \quad {\rm and} \quad  \CJ/\langle s, \, x,\,y-\alpha \rangle \cong \mK,
	\end{equation*}
	which are domains. Hence every nonzero prime ideal of $\CJ$ is completely prime. 
	
	(3) The zero ideal $\langle 0 \rangle$ is locally closed in $\Spec(\CJ)$, and hence primitive. The ideals $\langle s,\,x, \,y-\alpha \rangle $ are maximal and therefore primitive. By contrast, the ideal $\langle s, x \rangle$ is not primitive, since $\CJ/\langle s, x \rangle \cong \mK[y]$ is not a primitive algebra.
	
	(4) This follows immediately from the inclusions among the prime ideals.
\end{proof}

\subsection{Simple $\CJ$-modules}
For an algebra $R$, let $\widehat{R}$ denote the set of isomorphism classes $[M]$ of simple $R$-modules $M$. 

\begin{corollary} \label{1Feb26} 
	Assume that $\mK$ is algebraically closed of characteristic $0$.
	\begin{enumerate}
		\item The set of  isomorphism classes of finite-dimensional simple $\CJ$-modules is 
		\begin{equation*}
		\widehat{\CJ}({\rm fin. \,dim.})= \{ [S_{\alpha}] \mid \alpha \in \mK \}, \quad \text{where }\,\, S_{\alpha}:=\CJ/\langle s, \, x, \, y-\alpha \rangle. 
		\end{equation*}
		In particular, every finite-dimensional simple $\CJ$-module is one-dimensional.
		
		\item A simple $\CJ$-module $M$ is infinite-dimensional if and only if it is faithful, equivalently, if and only if the element $s$ acts bijectively on $M$. In this case, $M$ is naturally a simple $\CJ_s$-module.
	\end{enumerate}	
\end{corollary}
\begin{proof}
	(1) Let $M$ be a finite-dimensional simple $\CJ$-module. Then its annihilator $\ann_{\CJ}$(M) is a maximal ideal. By Theorem \ref{A24Jan26}(4), $\ann_{\CJ}(M)=\langle s,\, x,\,y-\alpha \rangle$ for some $\alpha \in \mK$. Since $\CJ/\langle s,\, x,\,y-\alpha \rangle \cong \mK$ is one-dimensional, it follows that $M \cong S_{\alpha}$. 
	
	(2) Let $M$ be a simple $\CJ$-module. Then $M$ is simple over $\CJ/\ann_{\CJ}(M)$.  By Theorem~\ref{A24Jan26}(3), the annihilator $\ann_{\CJ}(M)$ is either zero or of the form $\langle s,\, x,\,y-\alpha \rangle$. In the latter case, $M$ is one-dimensional. Thus $M$ is infinite-dimensional if and only if $\ann_{\CJ}(M)=\langle 0  \rangle$, that is, if and only if $M$ is faithful. 
	Since $s$ is normal, it acts either by zero or bijectively on any simple module. Hence  \(\ann_{\CJ}(M)=\langle 0\rangle\) if and only if $s$ acts bijectively on $M$. In this case, $M$ extends to a simple $\CJ_s$-module.
\end{proof}

\begin{remark}
		A classification of the finite-dimensional simple $\CJ$-modules was previously obtained in \cite[Theorem 2.6]{Andruskiewitsch-Bagio-Della Flora-Flores-1} by different methods.		
		Corollary \ref{1Feb26} shows that the infinite-dimensional representation theory of $\CJ$ is governed by the simple algebra $\CJ_s$. Indeed, every infinite-dimensional simple $\CJ$-module extends uniquely to a simple $\CJ_s$-module, yielding an injection $$\widehat{\CJ}({\rm infin. \,dim.}) \hookrightarrow \widehat{\CJ_s}.$$
		Via the isomorphism $\CJ_s\cong M_2(\CA)$ and Morita equivalence, these modules correspond to a subclass of simple $\CA$-modules. Moreover, by \cite[Lemma 3.1]{Bavula-Oystaeyn},
		the localization functor induces a bijection
		\begin{equation*}
		\widehat{\CJ}({\rm infin. \,dim.}) \simeq \widehat{\CJ_s}(\CJ\text{-socle}), \quad [M] \mapsto[\CJ_s \otimes_{\CJ} M]
		\end{equation*}
		whose inverse is given by the socle functor ${\rm Soc}: [N] \mapsto [{\rm Soc}_{\CJ} (N)]$, where $\widehat{\CJ_s}(\CJ\text{-socle})$ consists of those simple $\CJ_s$-modules whose restriction to $\CJ$ has nonzero socle. 
\end{remark}

We present a family of infinite-dimensional simple $\CJ$-modules $M_{\alpha}$, indexed by $\alpha \in \mK^*$.
As a vector space $M_{\alpha}=\bigoplus_{i\in\N} \mK e_i$. The action of the generators $x,s,y$ of $\CJ$ on $M_{\alpha}$ is given by
\begin{equation*}
x e_i= \begin{cases}
0,  &i \,\, \text{even}\\ 
e_{i-1}, &i \,\, \text{odd}
\end{cases}, \quad s e_i= \begin{cases}
\alpha e_{i}-\frac{i}{2} e_{i-2},  &i \,\, \text{even}\\ 
\alpha e_{i}-\frac{i-1}{2} e_{i-2}, &i \,\, \text{odd}
\end{cases},  \quad y e_i= \begin{cases}
\alpha e_{i+1}-\frac{i}{2} e_{i-1},  &i \,\, \text{even}\\ 
\alpha e_{i+1}-\frac{i-1}{2} e_{i-1}, &i \,\, \text{odd}
\end{cases}. 
\end{equation*}
Here we adopt the convention $e_{-2}=e_{-1}=0$. A straightforward verification shows that these formulas satisfy the defining relations of $\CJ$, and thus define a $\CJ$-module structure on $M_{\alpha}$. The actions of $x$ and $s-\alpha$ are illustrated in the diagram below.
\begin{equation*}
\begin{aligned}
\begin{tikzpicture}[>=Stealth, scale=0.9, every node/.style={scale=0.9}]

\node (e0) at (0,0.8) {$e_0$};
\node (e2) at (1.6,0.8) {$e_2$};
\node (e4) at (3.2,0.8) {$e_4$};
\node (A)  at (4.6,0.8) {$\cdots$};
\node (e2i) at (6.0,0.8) {$e_{2i}$};
\node (e2ip2) at (7.6,0.8) {$e_{2i+2}$};
\node (C) at (9.0,0.8) {$\cdots$};

\node (e1) at (0,-0.8) {$e_1$};
\node (e3) at (1.6,-0.8) {$e_3$};
\node (e5) at (3.2,-0.8) {$e_5$};
\node (B)  at (4.6,-0.8) {$\cdots$};
\node (e2ip1) at (6.0,-0.8) {$e_{2i+1}$};
\node (e2ip3) at (7.6,-0.8) {$e_{2i+3}$};
\node (D) at (9.0,-0.8) {$\cdots$};

\draw[->] (e1) -- node[right] {$x$} (e0);
\draw[->] (e3) -- node[right] {$x$} (e2);
\draw[->] (e5) -- node[right] {$x$} (e4);
\draw[->] (e2ip1) -- node[right] {$x$} (e2i);
\draw[->] (e2ip3) -- node[right] {$x$} (e2ip2);

\draw[->] (e2) -- node[below] {\scriptsize $-1$} node[above] {\scriptsize $s-\alpha$} (e0);
\draw[->] (e4) -- node[below] {\scriptsize $-2$} node[above] {\scriptsize $s-\alpha$} (e2);
\draw[->] (A) -- (e4);
\draw[->] (e2i) -- node[below] {\scriptsize $-i$} node[above] {\scriptsize $s-\alpha$} (A);
\draw[->] (e2ip2) -- node[below] {\scriptsize $-i-1$} node[above] {\scriptsize $s-\alpha$} (e2i);
\draw[->] (C) -- (e2ip2);

\draw[->] (e3) -- node[above] {\scriptsize $-1$} node[below] {\scriptsize $s-\alpha$} (e1);
\draw[->] (e5) -- node[above] {\scriptsize $-2$} node[below] {\scriptsize $s-\alpha$} (e3);
\draw[->] (B) -- (e5);
\draw[->] (e2ip1) -- node[above] {\scriptsize $-i$} node[below] {\scriptsize $s-\alpha$} (B);
\draw[->] (e2ip3) -- node[above] {\scriptsize $-i-1$} node[below] {\scriptsize $s-\alpha$} (e2ip1);
\draw[->] (D) -- (e2ip3);

\end{tikzpicture}
\end{aligned}
\end{equation*}

\begin{proposition}
Assume ${\rm char}\,\mK=0$.	For any $\alpha \in \mK^*$, the $\CJ$-module $M_{\alpha}$ is simple and faithful. 
\end{proposition}
\begin{proof}
	Since $\alpha\neq 0$,  the action of $y$ shows that $M_{\alpha}$ is cyclic with generator $e_0$.
	Write $M_{\alpha}=M[0]\oplus M[1]$, where $M[0]=\bigoplus_{i\in \N} \mK e_{2i}$ and $M[1]=\bigoplus_{i\in \N}\mK e_{2i+1}$. Then $xM[0]=0$ and $xM[1]=M[0]$. 
	
	 Let $V$ be a nonzero submodule of $M_{\alpha}$.  We first show that $V \cap M[0] \neq 0$. Take a nonzero element $v=v_0 +v_1 \in V$ with $v_0\in M[0]$ and $v_1 \in M[1]$. If $v_1=0$ there is nothing to prove. Otherwise, $xv=xv_1 \in V\cap M[0]$ is nonzero. Now let $0 \neq v \in V \cap M[0]$. Write $v=\sum_{i=0}^m \mu_i e_{2i}$ where $\mu_i \in \mK$ with $\mu_m \neq 0$. If $m>0$ then $(s-\alpha)v=-\sum_{i=0}^m i \mu_i e_{2i-2}$ is a nonzero element in $V \cap M[0]$. Iterating this process gives $(s-\alpha)^m v \in \mK^* e_0$, and hence $e_0 \in V$.  Since $M_{\alpha}$ is generated by $e_0$, it follows that $V=M_{\alpha}$. Thus $M_{\alpha}$ is simple. 	 
	 Finally, Corollary~\ref{1Feb26}(2) shows that every infinite-dimensional simple $\CJ$-module is faithful. Hence $\ann_{\CJ}(M_{\alpha})=\langle 0\rangle$.
\end{proof}

\section{The super Jordan plane in positive characteristic}  \label{SupJordan-p} 

In this section we study the super Jordan plane $\CJ$ over a field of positive characteristic. Our main results include descriptions of the center, the PI degree, and the classical ring of quotients of $\CJ$, as well as a determination of its prime spectrum, Azumaya locus, and simple modules.

\subsection{Center, PI degree and ring of quotients}
Recall that in the localization $\CJ_s$ we write $y'=ys^{-1}$ and $e=y'x$, where $e$ is an idempotent. The following lemma records a $p$th-power identity needed for the computation of the center of $\CJ$.
\begin{lemma} \label{22Jan26} 
	Assume ${\rm char}\,\mK=p >2$. Then $y^{\prime \,2p}=(ys^{-1})^{2p}=y^{2p} s^{-2p}$. 
\end{lemma}
\begin{proof}
	First observe that $y^{\prime\, 2}=(ys^{-1})^2=(y^2s^{-1}+ys^{-1}x)s^{-1}=(y^2s^{-1}+e)s^{-1}$.  Set $H:=y^2s^{-1}+e$. Then a direct computation shows that  $s^{-1}H=(H+1)s^{-1}$. It follows that
	\begin{equation*}
	y^{\prime \,2p}=(Hs^{-1})^p=H(H+1)\cdots (H+p-1)s^{-p}. 
	\end{equation*}
	Since ${\rm char}\,\mK=p$, we have the polynomial identity $\prod_{i=0}^{p-1}(H+i)=H^p-H$, and therefore,
	\begin{equation*}
	y^{\prime \,2p}=(H^p-H)s^{-p}. 
	\end{equation*}
	Next note that $[y^2s^{-1}, e]=0$ and that $e$ is idempotent. Therefore,
	\begin{equation*}
    H^p=(y^2s^{-1}+e)^p=(y^2s^{-1})^p+e, 
	\end{equation*}
	  which gives
	\begin{equation*}
	H^p-H=(y^2 s^{-1})^p-y^2s^{-1}. 
	\end{equation*}
	Since $[s^{-1}, y^2]=1$, the subalgebra generated by $y^2$ and $s^{-1}$ is isomorphic to the Weyl algebra $A_1(\mK)$. Applying \eqref{normp} gives $(y^2s^{-1})^p=y^{2p}s^{-p}+y^2s^{-1}$. Thus $H^p-H=y^{2p}s^{-p}$, and substituting back yields $y^{\prime \,2p}=y^{2p}s^{-2p}$, as claimed. 	
\end{proof}

\begin{remark}
	Assume ${\rm char}\,\mK=p>2$. Then in the algebras $\CJ_s$ and $\CJ$ we have
	\begin{equation*}
	\prod_{i=0}^{2p-1}(y+ix)=y^{2p}. 
	\end{equation*}
	Indeed, using the relation $s^{-1}y=(y+x)s^{-1}$, we can write $y^{\prime\, 2p}=(ys^{-1})^{2p}=\prod_{i=0}^{2p-1}(y+ix) \cdot s^{-2p}$. On the other hand, Lemma~\ref{22Jan26} gives $y^{\prime \,2p}=y^{2p} s^{-2p}$. Comparing the two expressions and cancelling $s^{-2p}$ yields the claimed identity.
\end{remark}

Recall that the \emph{PI degree} of a prime PI ring $R$, denoted $\operatorname{PI-deg} R$, is defined as the integer $d$ such that $d^2$ is the dimension of the central simple quotient ring $Q(R)$ over its center $Z(Q(R))$. 
If $R$ is moreover an affine $\mK$-algebra with $\mK$ algebraically closed, then $\operatorname{PI-deg} R$ coincides with the maximal dimension over $\mK$ of the simple $R$-modules. A prime ideal $\gp$ of a prime PI ring $R$ is \emph{regular} if $\operatorname{PI-deg} R/\gp=\operatorname{PI-deg} R$.

\begin{proposition} \label{23Jan26}  
	Assume ${\rm char}\,\mK=p >2$. Then the following hold:
	\begin{enumerate}
		\item The centers of $\CJ_s$ and $\CJ$ are given, respectively, by
		\begin{equation*}
		Z(\CJ_s)=\mK[s^{\pm p}, y^{2p}], \qquad Z(\CJ)=\mK[s^p, y^{2p}].
		\end{equation*}
		\item $\CJ_s$ is an Azumaya algebra of rank $(2p)^2$ over its center. 
		\item $\CJ_s$ is a prime PI ring whose prime ideals are all regular, and $\operatorname{PI-deg}\,\CJ_s=2p$. 
 		\item There is a bijective between the ideals of $\CJ_s$ and the ideals of $Z(\CJ_s)$ given by $I \mapsto I \cap Z(\CJ_s)$ with inverse $\ga \mapsto \ga \CJ_s$. Moreover, this correspondence preserves primeness. 
		\item Every simple $\CJ_s$-module has dimension $2p$. 
		\item $\CJ$ is a prime PI algebra, with $\operatorname{PI-deg}\,\CJ=2p$. 
		\item Let $Z:=Z(\CJ)=\mK[s^p, y^{2p}]$, and $K:=\Frac(Z)=\mK(s^p, y^{2p})$. Then 
		\begin{equation*}
		Q(\CJ) \cong \CJ \otimes_ZK \cong M_2(\CA)\otimes_Z K. 
		\end{equation*}
		is a central simple $(2p)^2$-dimensional algebra over the field $K$. 
	\end{enumerate}
\end{proposition}
\begin{proof}
	(1) Since $Z(\CA)=\mK[s^{\pm p}, y^{\prime \,2p}]$, Lemma \ref{12Jan26} implies  $Z(\CJ_s)=\mK[s^{\pm p}, y^{\prime \, 2p}]$. By Lemma \ref{22Jan26}, $y^{\prime\, 2p}=y^{2p}s^{-2p}$, so $Z(\CJ_s)=\mK[s^{\pm p}, y^{2p}]$. Intersecting with $\CJ$ gives $Z(\CJ)=\CJ \cap Z(\CJ_s)=\mK[s^p, y^{2p}].$
	
	(2) To check that $\CJ_s$ is an Azumaya algebra, it suffices to check this holds after an extension of scalars to $\overline{\mK}$, so without loss of generality, we can assume that $\mK$ is algebraically closed. The maximal ideals of $Z(\CJ_s)$ are $\gm_{\alpha,\beta}=\langle s^p-\alpha, \,y^{2p}-\beta \rangle$ where $\alpha \in \mK^*$ and $\beta \in \mK$. Using $\CJ_s \cong M_2(\CA)$, we obtain
	\begin{equation} \label{CJsm}  
	\CJ_s/\gm_{\alpha,\beta}\CJ_s \cong M_2\big( \CA/\gm_{\alpha,\beta}\CA\big) \cong M_2(M_p(\mK)) \cong M_{2p}(\mK). 
	\end{equation}
	Thus every fibre over a maximal ideal of the center is a full matrix algebra, and therefore $\CJ_s$ is an Azumaya algebra of rank $(2p)^2$. 
	
	(3) Since $\CJ_s \cong M_2(\CA)$ is prime and Azumaya, the statement follows from the Artin--Procesi theorem (see \cite[Theorem 13.7.14]{MR}). 
	
	(4) This is a standard property of Azumaya algebras; see \cite[Proposition 13.7.9]{MR}. 
	
	(5) Since $\CJ_s$ is Azumaya of rank $(2p)^2$,  every simple $\CJ_s$-module has dimension $2p$. 
	
	(6) The algebra $\CJ$ is a finitely generated module over its center, hence a PI ring by \cite[Corollary 13.1.13]{MR}. If fact, from the PBW basis, $\CJ$ is a free $Z(\CJ)$-module of rank $4p^2$. Since $\CJ_s$ is prime and $s$ is a regular normal element, the zero ideal of $\CJ$ is prime. Moreover, $\CJ$ and $\CJ_s$ have the same classical  quotient ring,  hence $\operatorname{PI-deg}\CJ=\operatorname{PI-deg}\CJ_s=2p$. 
	
	(7) Since $\CJ$ is a prime PI ring, by Posner's Theorem (\cite[Theorem 13.6.5]{MR}), the classical quotient ring $Q(\CJ)$ is obtained by inverting the nonzero central elements of $\CJ$, and $Q(\CJ)$ is a central simple algebra with center $K$. Since $\operatorname{PI-deg}\CJ=2p$, we have $\dim_K Q(\CJ) =(2p)^2$. 
\end{proof}

\subsection{Prime ideals and the Azumaya locus}

The following theorem gives an explicit description of the prime, primitive, and maximal ideals of the super Jordan plane $\CJ$ in positive characteristic.
\begin{theorem} \label{28Jan26} 
	Assume that $\mK$ is algebraically closed of characteristic $p>2$.
	\begin{enumerate}
		\item The prime spectrum of $\CJ$ is given by
	\begin{equation*}
	\Spec(\CJ)= \{ \langle s, \,x \rangle\} \,\, \cup \,\,\{\langle s,\,x, \,y-\alpha \rangle \mid \alpha \in \mK   \} \,\, \cup \,\, \{ \ga\CJ_s \cap \CJ \mid \ga \in \Spec(\mK[s^{\pm p}, y^{2p}]) \}. 
	\end{equation*}
	
    \item The primitive ideals of $\CJ$ coincide with its maximal ideals and are given by
	\begin{equation*}
	\Prim(\CJ)=\Max(\CJ)= \{\langle s,\,x, \,y-\alpha \rangle \mid \alpha \in \mK   \} \,\cup \, \{ \langle s^p-\alpha, \,y^{2p}-\beta \rangle \mid \alpha \in \mK^*,\, \beta\in \mK\}. 
	\end{equation*}
	The corresponding primitive factor algebras are as follows:
	\begin{enumerate}
		\item For each $\alpha \in \mK$, $\CJ/\langle s, \, x, \, y-\alpha \rangle \cong \mK$. 
		\item For each $\alpha\in \mK^*$ and $\beta\in \mK$, $\CJ/\langle s^p-\alpha,\,y^{2p}-\beta \rangle \cong M_{2p}(\mK)$. 
	\end{enumerate}
		\end{enumerate}
\end{theorem}
\begin{proof}
	(1) Let $\Spec(\CJ, s)$ (resp. $\Spec_s(\CJ)$) denote the set of prime ideals of $\CJ$ containing (resp. not containing) $s$. Then $\Spec(\CJ)=\Spec(\CJ, s) \cup \Spec_s(\CJ)$. 	
	Prime ideals of $\CJ$ containing $s$ correspond bijectively to prime ideals of the factor algebra $\Lambda=\CJ/s\CJ$. By Proposition \ref{24Jan26}, we have
	\begin{equation*}
	\Spec(\CJ, s)= \{ \langle s, \,x \rangle\} \,\, \cup \,\,\{\langle s,\,x, \,y-\alpha \rangle \mid \alpha \in \mK   \}. 
	\end{equation*}
	Since $s$ is normal in $\CJ$, prime ideals of $\CJ$ not containing $s$ are in bijection with prime ideals of the localization $\CJ_s$. By Proposition \ref{23Jan26}(4), the prime ideals of $\CJ_s$ are precisely those extended from its center $Z(\CJ_s)=\mK[s^{\pm p}, y^{2p}]$, that is, $\Spec(\CJ_s)=\{ \ga \CJ_s \mid \ga \in \Spec(\mK[s^{\pm p}, y^{2p}]) \}$. Consequently, $\Spec_s(\CJ)=\{ \ga \CJ_s \cap \CJ \mid \ga \in \Spec(\mK[s^{\pm p}, y^{2p}]) \}$. Combining these two cases gives the stated description of $\Spec(\CJ)$. 
	
	(2) Since $\CJ$ is a PI algebra, every primitive ideal is maximal. It follows from statement (1) that 
	\begin{equation*}
	\Max(\CJ)= \{\langle s,\,x, \,y-\alpha \rangle \mid \alpha \in \mK   \} \,\, \cup \,\, \{ \gm\CJ_s \cap \CJ \mid \gm \in \Max(\mK[s^{\pm p}, y^{2p}]) \}. 
	\end{equation*}
	Since $\mK$ is algebraically closed, the maximal ideals of $\mK[s^{\pm p}, y^{2p}]$ are precisely $\gm_{\alpha,\beta}=\langle s^p-\alpha, \,y^{2p}-\beta \rangle$ where $\alpha \in \mK^*$ and $\beta \in \mK$. We claim that $\gm_{\alpha,\beta}\CJ_s \cap \CJ=\gm_{\alpha,\beta}\CJ$. Indeed, since $\alpha \in \mK^*$, the image of $s$ is invertible in $\CJ/\gm_{\alpha,\beta}\CJ$, and hence, 
	\begin{equation*}
	\CJ/\gm_{\alpha,\beta}\CJ \cong \CJ_s/\gm_{\alpha,\beta}\CJ_s. 
	\end{equation*}
	By \eqref{CJsm}, the latter algebra is isomorphic to $M_{2p}(\mK)$,  showing that $\gm_{\alpha,\beta}\CJ$ is maximal in $\CJ$ and proving the claim. The description of the primitive factor algebras follows immediately. 
\end{proof}

\begin{remark}
	Let $R$ be an algebra finite over its center $Z=Z(R)$. The \emph{Azumaya locus} of $R$ over $Z$ is the dense open subset of $\Max(Z)$ defined by
	\begin{equation*}
    \operatorname{Az}(R):=\{ \gm \in \Max(Z) \mid R_{\gm} \text{ is Azumaya over } Z_{\gm} \},
	\end{equation*} 
	where $R_{\gm}$ and $Z_{\gm}$ denote the localizations of $R$ and $Z$ at $\gm$, respectively. Equivalently, a maximal ideal $\gm \in \Max(Z)$ lies in $\operatorname{Az}(R)$ if and only if $R/\gm R$ is a central simple algebra over the field $Z/\gm$. On the Azumaya locus, the PI degree of $R/\gm R$ achieves the PI degree of $R$. Thus, $\gm \in \operatorname{Az}(R)$ if and only if $\gm R$ is the annihilator of a simple $R$-module $M$ with $\dim M=\operatorname{PI-deg}R$. 
	
	Assume that $\mK$ is algebraically closed of characteristic $p>2$. Then the Azumaya locus of $\CJ$ is 
	\begin{equation*}
	\operatorname{Az}(\CJ)= \{\gm_{\alpha,\beta} \mid \alpha \in \mK^*,\, \beta\in \mK \}, \quad {\rm where}\,\, \gm_{\alpha,\beta}=\langle s^p-\alpha, \,y^{2p}-\beta \rangle \in \Max(\mK[s^p, y^{2p}]). 
	\end{equation*}
	For each $\gm_{\alpha,\beta} \in \operatorname{Az}(\CJ)$, the corresponding factor algebra satisfies $\CJ/\gm_{\alpha,\beta}\CJ \cong M_{2p}(\mK)$, and is therefore a central simple algebra over $\mK$ of PI degree $2p$. The complement of the Azumaya locus is the closed subset $V(s):= \{ \langle s^p, y^{2p}-\beta \rangle \mid \beta\in \mK \}$. For these maximal ideals, the factor algebras $\CJ/\gm \CJ$ have PI degree one, and $\CJ$ is not Azumaya at $\gm$.  
\end{remark}

\subsection{Classification of simple $\CJ$-modules}

Suppose that ${\rm char}\,\mK=p>2$. Then $\CJ$ is a PI algebra, and hence every simple $\CJ$-module is finite-dimensional. Moreover, by the description of the primitive factors of $\CJ$ in Theorem \ref{28Jan26}(2), every simple $\CJ$-module has dimension either $1$ or $2p$. The following theorem gives an explicit construction and classification of all simple $\CJ$-modules. 
\begin{theorem} \label{29Jan26}  
	Let $\mK$ be an algebraically closed field of characteristic $p>2$.
	\begin{enumerate}
		\item For any $\alpha \in \mK$, the $\CJ$-module $S_{\alpha}:=\CJ/\langle s, \, x, \, y-\alpha \rangle$ is simple. 
		\item For $\alpha \in \mK^*$ and $\beta \in \mK$, the $\CJ$-module
		\begin{equation*}
		V_{\alpha,\beta}:=\CJ/\CJ(x, \, s-\alpha,\,y^{2p}-\beta)
		\end{equation*}
		is simple of dimension $2p$. 
		\item The set of isomorphism classes of simple $\CJ$-modules is 
		\begin{equation*}
		\widehat{\CJ}=\{ [S_{\alpha}] \mid \alpha \in \mK \} \,\cup \, \{ [V_{\alpha, \beta}] \mid \alpha \in \mK^*, \beta \in \mK \}. 
		\end{equation*}
	\end{enumerate}
\end{theorem}
\begin{proof}
	(1) This is immediate since $\dim S_{\alpha}=1$. 
	
	(2)  Let $\xi:=1+\CJ(x, \, s-\alpha,\,y^{2p}-\beta)$ be the canonical generator of $V_{\alpha,\beta}$. By the PBW basis of $\CJ$,
	\begin{equation*}
	V_{\alpha,\beta}=\bigoplus_{i=0}^{2p-1}\mK y^i \xi=V_0 \oplus V_1, 
	\end{equation*} 
	where 
	\begin{equation*}
	V_0:=\bigoplus_{i=0}^{p-1} \mK y^{2i} \xi, \qquad V_1:=\bigoplus_{i=0}^{p-1}\mK y^{2i+1} \xi. 
	\end{equation*}
	
	We first show that $x V_0=0$ and $xV_1 =V_0$. Since $xy^2=(y^2-s)x$, it follows that 
	\begin{equation} \label{xy2i} 
	xy^{2i}=(y^2-s)^ix. 
	\end{equation}
	Thus $xy^{2i}\xi=(y^2-s)^ix \xi=0$, and therefore $x V_0=0$. Next we compute
		\begin{equation*}
	\begin{aligned}
	x y^{2i+1}\xi =xy^{2i}y\xi \stackrel{\eqref{xy2i}}{=}(y^2-s)^i xy \xi 
	= (y^2-s)^i(s-yx)\xi
	=\alpha (y^2-s)^i \xi. 
	\end{aligned}
	\end{equation*}
	Using \eqref{y2s}
	we see that $xy^{2i+1}\xi \in V_0$ with leading term $\alpha y^{2i}\xi$. Since $\alpha\neq 0$, it follows that $x V_1=V_0$.

    Let $W$ be a nonzero submodule of $V_{\alpha,\beta}$. We claim that $W \cap V_0 \neq 0$.  Take a nonzero element $v=v_0+v_1\in W$ with $v_0\in V_0$ and $v_1 \in V_1$. If $v_1=0$ there is nothing to prove. Otherwise, $xv=x(v_0+v_1)=xv_1 \in W \cap V_0$ is nonzero. 
    
    Now take a nonzero $v\in W \cap V_0$ and write $v=\sum_{i=0}^{m}\mu_i y^{2i}\xi$, where $\mu_m\in \mK^*$. Since $sy^{2i}\xi=(y^2-s)^i s\xi=\alpha(y^2-s)^i \xi$, using \eqref{y2s} we obtain 
    \begin{equation*}
    (s-\alpha) y^{2i}\xi=\alpha\big((y^2-s)^i-y^{2i}\big)\xi=\alpha^* y^{2(i-1)}\xi+\sum_{j=0}^{i-2} \l_j y^{2j} \xi
    \end{equation*}
    for some $\alpha^*\in\mK^*$ and $\l_j\in \mK$.
    Thus $s-\alpha$ acts triangularly on $V_0$ with respect to the basis $\{y^{2i}\xi\}$ and strictly lowers degree.    
   Applying $(s-\alpha)^m$ to $v$ produces a nonzero scalar multiple of $\xi$, so $\xi \in W$ and hence $W = V_{\alpha,\beta}$. This proves that $V_{\alpha,\beta}$ is simple of dimension $2p$.

    (3) Let $M$ be a simple $\CJ$-module. Since $s$ is a normal element of $\CJ$, either $sM=0$ or $s$ acts bijectively on $M$. If $sM=0$, then by Theorem \ref{28Jan26}(2) the annihilator of $M$ must be of the form $\langle s, \,x,\, y-\alpha\rangle$ for some $\alpha\in \mK$, and hence $M \cong S_{\alpha}$. Suppose now that $s$ acts bijectively on $M$. Since $xs=sx$ and $\mK$ is algebraically closed, the commuting endomorphisms $x$ and $s$ admit a common eigenvector $\xi\in M$.  As $x^2=0$ and $s$ is invertible on $M$, we must have $x \xi=0$ and $s \xi=\alpha \xi$ for some $\alpha \in \mK^*$. By Schur's Lemma, the central element $y^{2p}$ acts on $M$ as a scalar, say $\beta \in \mK$. Since $M$ is simple, it is generated by $\xi$, and hence there is a surjective $\CJ$-module homomorphism $\pi: V_{\alpha, \beta} \twoheadrightarrow M$. By statement (2), $V_{\alpha,\beta}$ is simple, so $\pi$ is an isomorphism and $M \cong V_{\alpha,\beta}$. 
\end{proof}

\

\

\noindent
\textbf{Data availability statement}  Data sharing is not applicable to this article as no datasets were generated or analysed during the current study.

\

\

\noindent
\textbf{Declarations} The author declares no conflict of interest.

\small{

\end{document}